\documentclass[11pt,a4paper]{amsart}

\usepackage{amsmath,amssymb,amsfonts,amsthm,bbm,microtype}

\DeclareMathOperator{\expectation}{\mathsf{E}}
\DeclareMathOperator{\var}{var}

\newtheorem{theorem}{Theorem}

\newtheorem{proposition}{Proposition}[theorem]
\newtheorem{corollary}{Corollary}[theorem]

\begin{document}
\title{On eventually always hitting points}
\date{\today}

\begin{abstract}
  We consider dynamical systems $(X,T,\mu)$ which have exponential
  decay of correlations for either H\"{o}lder continuous functions or
  functions of bounded variation. Given a sequence of balls
  $(B_n)_{n=1}^\infty$, we give sufficient conditions for the set of
  eventually always hitting points to be of full measure. This is the
  set of points $x$ such that for all large enough $m$, there is a $k
  < m$ with $T^k (x) \in B_m$.  We also give an asymptotic estimate as
  $m \to \infty$ on the number of $k < m$ with $T^k (x) \in B_m$. As
  an application, we prove for almost every point $x$ an asymptotic
  estimate on the number of $k \leq m$ such that $a_k \geq m^t$, where
  $t \in (0,1)$ and $a_k$ are the continued fraction coefficients of
  $x$.
\end{abstract}

\author{Charis Ganotaki}
\email{ch.ganotaki@gmail.com}
  
\author{Tomas Persson}
\address{Centre for Mathematical Sciences, Lund University, Box 118, 221 00 Lund, Sweden}
\email{tomasp@maths.lth.se}

\maketitle

\section{Introduction}

Let $T \colon X \to X$ be a dynamical system with an invariant
probability measure $\mu$. We assume that $X$ is a metric space and
that $\mu$ is ergodic.

If $(B_n)_{n=1}^\infty$ is a sequence of balls, it is
often of interest to know whether or not for $\mu$ almost every $x$,
the iterate $T^n (x)$ belongs to $B_n$ for infinitely many $n$. Let us
call the set of such points the set of infinitely often hitting
points,
\[
H_\mathrm{io} = H_\mathrm{io} ((B_n)_{n=1}^\infty) = \{\, x \in X : T^n
(x) \in B_n \text{ for infinitely many } n \,\}.
\]
In some cases, it is possible to prove that $\mu (H_\mathrm{io})$ is
either $0$ or $1$, depending on if $\sum_n \mu (B_n)$ is convergent or
divergent. Such and similar results are called \emph{dynamical
  Borel--Cantelli lemmata} \cite{Haydnetal, Kim, Xing}.

Note that $x \in H_\mathrm{io}$ if and only if
\[
\sum_{k = 1}^\infty \mathbbm{1}_{B_k} (T^k (x)) = \infty.
\]
When $\sum \mu (B_n) = \infty$, it is sometimes possible to prove the
more precise statement, that for $\mu$ almost all points $x$ holds
\[
\lim_{n \to \infty} \frac{\sum_{k = 1}^n \mathbbm{1}_{B_k} (T^k
  (x))}{\sum_{k=1}^n \mu (B_k)} = 1.
\]
Results of this type are known as \emph{strong (dynamical)
  Borel--Cantelli lemmata} \cite{Haydnetal, Kim, Xing}. Hence, a
strong Borel--Cantelli lemma gives us the asymptotic rate at which the
iterate $T^k (x)$ hits the ball $B_k$.

In this paper we shall investigate a set which is closely related to
$H_\mathrm{io}$, namely the set of \emph{eventually always hitting
  points}. This is the set $H_\mathrm{ea}$ of points $x$ such that for
all large enough $m$, there exists a $k \in \{0,1, \ldots, m-1\}$ such
that $T^k (x) \in B_m$. In other words, we have
\begin{align*}
  H_\mathrm{ea} &= \{ \, x \in X : \exists n \text{ s.t. } \forall m
  > n, \exists k \in [0,m-1] \text{ s.t. } T^k (x) \in B_m \,\}
  \\ &= \bigcup_{n=1}^\infty \bigcap_{m = n}^\infty \bigcup_{k =
    0}^{m-1} T^{-k} B_m.
\end{align*}

The name \emph{eventually always hitting points} was introduced by
Kelmer \cite{Kelmer}. However, the Hausdorff dimension of
$H_\mathrm{ea}$ was studied earlier by Bugeaud and Liao
\cite{BugeaudLiao}.

If $\mu (B_n) \to 0$, then $H_\mathrm{ea} \subset H_\mathrm{io}$
modulo a set of measure zero \cite{KirsebomKundePersson}. In this
sense, $H_\mathrm{ea}$ is a smaller set than $H_\mathrm{io}$.

There has recently been several results about the measure of the set
$H_\mathrm{ea}$. Kelmer proved necessary and sufficient conditions for
$\mu (H_\mathrm{ea}) = 1$ in the setting of flows on hyperbolic
manifolds \cite{Kelmer}, and for flows on homogeneous spaces together
with Yu \cite{KelmerYu}. In the general setting, Kelmer proved
\cite{Kelmer} that if $c < 1$ and $\mu (B_n) \leq \frac{c}{n}$ for
infinitely many $n$, then $\mu (H_\mathrm{ea}) = 0$. Other recent
results can be found in the references \cite{KleinbockWadleigh,
  KelmerOh}.

Kleinbock, Konstantoulas and Richter
\cite{KleinbockKonstantoulasRichter} gave for some systems, including
the the Gau\ss-map, conditions that guarantee that either $\mu
(H_\mathrm{ea}) = 0$ or $\mu (H_\mathrm{ea}) = 1$. For the Gau\ss-map
and balls $B_n$ with centres in $0$, they proved that if $\mu (B_n)
\geq \frac{c \log \log n}{n}$ with $c > \frac{1}{\log 2}$, then $\mu
(H_\mathrm{ea}) = 1$, and if $\mu (B_n) \leq \frac{c \log \log n}{n}$
with $c < \frac{1}{\log 2}$, then $\mu (H_\mathrm{ea}) = 0$.

At about the same time, results related to those of Kleinbock,
Konstantoulas and Richter were obtained by Kirsebom, Kunde and Persson
\cite{KirsebomKundePersson}. Their results hold for a wider class of
dynamical systems, but they had to impose stronger conditions on $\mu
(B_n)$ in order to be able to conclude that the measure of
$H_\mathrm{ea}$ is one or zero: For piecewise expanding maps including
the Gau\ss-map, and for some quadratic maps, they proved that if $\mu
(B_n) \geq \frac{c (\log n)^2}{n}$ where $c$ is sufficiently large,
then $\mu (H_\mathrm{ea}) = 1$. The proof of this result uses heavily
that the dynamical system has exponential decay of correlations for
bounded variation against $L^1$ (and in fact, the result is valid for
any such dynamical system). Since exponential decay of correlations
for bounded variation against $L^1$ is only available for certain
one-dimensional systems, it would be interesting to prove a smilar
result under a more common type of correlation decay. This is one of
the goals of this paper.

We will consider dynamical systems which has exponential decay of
correlations for H\"{o}lder continuous functions. This includes for
instance expanding maps, Anosov diffeomorphisms and some
billiards. For such systems, we will prove under some extra regularity
conditions on $\mu$ that $\mu (B_n) \geq n^{-1} (\log n)^{2 +
  \varepsilon}$ for some $\varepsilon > 0$ implies that $\mu
(H_\mathrm{ae}) = 1$. Hence, we impose a sligthly stronger condition
on $\mu (B_n)$ compared to Kirsebom, Kunde and Persson, in order to
conclude that $H_\mathrm{ea}$ has full measure, but the result holds
for many new systems.

To prove these results, we follow the approach of Kleinbock,
Konstantoulas and Richter and consider the sum
\[
Z_m (x) = \sum_{k=0}^{m-1} \mathbbm{1}_{B_m} (T^k x).
\]
Then $x \in H_\mathrm{ea}$ if and only if $\liminf_{m \to \infty} Z_m
(x) \geq 1$.  To obtain that $\mu (H_\mathrm{ea}) = 1$, we will prove
that the expectation
\begin{equation} \label{eq:expectation}
  \expectation \biggl( \biggl( \frac{Z_m}{\expectation Z_m} - 1
  \biggr)^2 \biggr)
\end{equation}
goes to zero, along a subsequence which is not too sparse.

By applying this method to interval maps of the same type as those
considered by Kirsebom, Kunde and Persson, we are able to improve
their result and prove that in this setting the bound $\mu (B_n) \geq
n^{-1} (\log n)^{1 + \varepsilon}$ is enough to conclude that
$H_\mathrm{ea}$ has full measure.

Our method of proof allows us to conclude a strong version of the
statement $\mu (H_\mathrm{ea}) = 1$ (compare with the strong
Borel--Cantelli lemmata mentioned above). Under some regularity
assumptions on $n \mapsto \mu (B_n)$ together with the assumption that
$\mu (B_n) \geq n^{-1 + \varepsilon}$, we prove (Theorem~\ref{the:BV2}
and \ref{the:Holder2}) that
\begin{equation} \label{eq:asymptotics}
  \lim_{m \to \infty} \frac{\sum_{k=0}^{m-1} \mathbbm{1}_{B_m} (T^k
    (x))}{m \mu (B_m)} = 1,
\end{equation}
holds for $\mu$ almost every $x$. This is done by proving that the
expectation in \eqref{eq:expectation} goes to zero as $m \to \infty$,
not only along a subsequence.

As mentioned above, the statement $\mu (H_\mathrm{ea}) = 1$ is
equivalent to the statement that $\liminf_{m \to \infty} Z_m (x) \geq
1$ for almost every $x$, whereas \eqref{eq:asymptotics} gives more
precise asymptotic information about the number of $k < m$ for which
$T^k (x) \in B_m$. Kelmer \cite[Theorem~2]{Kelmer} and Kelmer--Yu
\cite[Theorem~1.8]{KelmerYu} gave in their settings sufficient
conditions for the quotient in \eqref{eq:asymptotics} to be bounded
between two positive constants.

\section{Results}

Let $B_n = B(x_n, r_n)$ be a nested sequence of balls. That is,
we assume that $B_{n+1} \subset B_n$ holds for all $n$. In
particular, $(r_n)_{n=1}^\infty$ is a decreasing sequence. All
our results on $H_\mathrm{ea}$ will be for this type of sequence
of balls. However, in some cases we will also assume that the
balls share the same centre, $x_n = x_0$.

We are going to give conditions on the dynamical system $(X,T,\mu)$
and on the balls that imply that $\mu (H_\mathrm{ea}) = 1$. In order
to do so, we will use estimates on decay of correlations. We consider
two cases.

The first and simplest case is when $X$ is an interval and $T$ has
exponential decay of correlations for bounded variation against
$L^1$. More precisely, we say that $(X,T,\mu)$ has exponential decay of
correlations for functions of bounded variation against $L^1$ if there
are constants $C, \tau > 0$ such that
\[
\biggl| \int f \circ T^n \cdot g \, \mathrm{d} \mu - \int f \,
\mathrm{d} \mu \int g \, \mathrm{d} \mu \biggr| \leq C e^{- \tau n}
\lVert f \rVert_1 \lVert g \rVert_{BV}
\]
holds for all $n \in \mathbbm{N}$ and all measurable functions
$f$ and $g$ (with $f \in L^1$ and $g$ of bounded variation). The
norms are defined by
\[
\lVert f \rVert_1 = \int |f| \, \mathrm{d} \mu, \qquad \lVert g
\rVert_{BV} = \sup |g| + \var g,
\]
where $\var g$ is the total variation of $g$ on $X$. A direct
consequence of this correlation decay is that
\begin{equation} \label{eq:intervalcorrelation}
| \mu (T^{-n} I \cap J) - \mu (I) \mu (J) | \leq 3C e^{-\tau n} \mu
(I),
\end{equation}
whenever $I$ and $J$ are intervals.

Exponential decay of correlations for bounded variation against $L^1$
holds for piecewise expanding interval maps \cite{LasotaYorke,
  Rychlik, Liveranietal} (including the Gau\ss-map) when $\mu$ is a
Gibbs measure, and for quadratic maps for Benedicks--Carleson
parameters \cite{Young}.

The second case is when $(X,T,\mu)$ has exponential decay of
correlations for H\"{o}lder continuous functions, in the sense that
there are constants $C, \tau > 0$ such that
\begin{equation} \label{eq:correlationhoelder}
  \biggl| \int f \circ T^n \cdot g \, \mathrm{d} \mu - \int f \,
  \mathrm{d} \mu \int g \, \mathrm{d} \mu \biggr| \leq C e^{- \tau n}
  \lVert f \rVert \lVert g \rVert
\end{equation}
holds for all $n$ and all H\"{o}lder continuous functions $f$ and
$g$. The norm is the $\alpha$-H\"{o}lder norm with $0 < \alpha
\leq 1$, defined by
\[
\lVert f \rVert = \sup |f| + \inf \{\, C : |f(x) - f(y)| \leq C
|x - y|^\alpha \ \forall x, y \,\}.
\]

Exponential decay of correlations for H\"{o}lder continuous functions
holds for instance for Anosov diffeomorphisms \cite{Bowen}, for
expanding maps, some billiards and other types of piecewise hyperbolic
systems \cite{Young2, Chernov} (for instance, but not only, when $\mu$
is a physical measure). All these systems also satisfy the other
assumptions in Theorem~\ref{the:Holder1} and \ref{the:Holder2} below.

\subsection{Results for interval maps}

We give below two results for interval maps. The first theorem is an
improvement of a similar result by Kirsebom, Kunde and Persson
\cite[Theorem~5]{KirsebomKundePersson}, giving a better lower bound on
$\mu (B_m)$ which implies full measure of $H_\mathrm{ea}$.  The second
theorem provides, in the same setting, an asymptotic result on the
number of $k < m$ such that $T^k (x) \in B_m$. The proofs of these
results are in Section~\ref{sec:BV}.

\begin{theorem} \label{the:BV1}
  Suppose that $(X,T,\mu)$ has exponential decay of correlations
  for bounded variation against $L^1$. Let $(B_m)_{m=1}^\infty$
  be a decreasing sequence of balls with centres in $x_0$ and
  assume that there is a number $m_0$ such that $\mu(B_m) \geq
  m^{-1} (\log m)^{1 + \varepsilon}$ holds for some $\varepsilon
  > 0$ and all $m \geq m_0$. Then $\mu(H_\mathrm{ea}) = 1$.
\end{theorem}

\begin{theorem} \label{the:BV2}
  Suppose that $(X,T,\mu)$ has exponential decay of correlations
  for bounded variation against $L^1$. Let $(B_m)_{m=1}^\infty$
  be a decreasing sequence of balls and assume that $m \mapsto m
  \mu (B_m)$ is increasing\footnote{With a decreasing function
    $f$, we mean a function such that $f(x) \geq f(y)$ whenever
    $x \leq y$. With an increasing function $f$, we mean a
    function such that $f(x) \leq f(y)$ whenever $x \leq y$.},
  and that there is a number $m_0$ such that $\mu(B_m) \geq m^{-1
    + \varepsilon}$ holds for some $\varepsilon > 0$ and all $m
  \geq m_0$. Then
  \[
  \lim_{m \to \infty} \frac{\# \{\, 0 \leq k < m : T^k (x) \in
    B_m \, \}}{m \mu (B_m)} = 1
  \]
  for $\mu$ almost every $x$.
\end{theorem}

The Gau\ss-map $G(x) = \frac{1}{x} \mod 1$ with $\mu$ given by
\[
\mu (E) = \frac{1}{\log 2} \int_E \frac{1}{1+x} \, \mathrm{d}x,
\]
satisfies the assumptions of Theorem~\ref{the:BV2}. This leads to
the following corollary about continued fraction expansions. Any
irrational number $x \in (0,1)$ can be written in a unique way as
a continued fraction
\[
x = \cfrac{1}{a_1 (x) + \cfrac{1}{a_2 (x) + \cfrac{1}{a_3 (x) +
      \ldots}}},
\]
where the numbers $a_1 (x), a_2 (x), \ldots$ are positive
integers and are called the continued fraction coefficients of
$x$. We have $a_k (x) = a$ if and only if $G^{k-1} (x) \in
(\frac{1}{a}, \frac{1}{a+1})$.

\begin{corollary}
  Let $t \in (0,1)$. Then for almost every $x \in (0,1)$ we have
  \[
  \lim_{m \to \infty} \frac{\# \{\, k \leq m : a_k (x) \geq m^t
    \,\}}{m^{1-t}} = \frac{1}{\log 2}.
  \]
\end{corollary}

\begin{proof}
  We consider an increasing sequence of positive integers $b_n$ and
  let $B_n = [0, b_n^{-1})$. Then $G^k (x) \in B_n$ if and only if
    $a_{k+1} (x) \geq b_n$. We have
    \[
    \mu (B_n) = \frac{1}{\log 2} \int_0^\frac{1}{b_n} \frac{1}{1 + x}
    \, \mathrm{d} x = \frac{\log (1 + \frac{1}{b_n})}{\log 2}.
    \]

    Let $b_n = [n^t]$.  Then, $n \mapsto \mu (B_n)$ is decreasing
    since $b_n$ is increasing, and $n \mapsto n \mu (B_n)$ is
    increasing for $n \geq n_0 (t)$. Theorem~\ref{the:BV2} then implies
    that
    \[
    \lim_{m \to \infty} \frac{\# \{\, k \leq m : a_k (x) \geq [m^t]
      \,\}}{m \log (1 + [m]^{-t})} = \frac{1}{\log 2}
    \]
    holds for almost all $x \in (0,1)$. Hence
    \[
    \lim_{m \to \infty} \frac{\# \{\, k \leq m : a_k (x) \geq [m^t]
      \,\}}{m^{1-t}} = \frac{1}{\log 2}
    \]
    for almost all $x$.

    The same argument for $b_n = \lceil n^t \rceil$ gives that
        \[
    \lim_{m \to \infty} \frac{\# \{\, k < m : a_k (x) \geq \lceil m^t
      \rceil \,\}}{m^{1-t}} = \frac{1}{\log 2}
    \]
    for almost all $x$. The statement in the corollary follows.
\end{proof}

\subsection{Results for systems with exponential decay of correlations for H\"{o}lder continuous functions}

Below we give two theorems corresponding to Theorems~\ref{the:BV1} and
\ref{the:BV2}, but assuming instead exponential decay of correlations
for H\"{o}lder continuous functions. The proofs are in
Section~\ref{sec:Holder}. They are similar to the previous proofs, but
the details are different and more involved.

\begin{theorem} \label{the:Holder1}
  Suppose $(X,T,\mu)$ has exponential decay of correlations for
  H\"{o}lder continuous functions and that the measure $\mu$
  satisfies
  \begin{equation} \label{eq:annulus1}
    \mu (B(x_0, r + \rho) \setminus B(x_0,r)) \leq C \rho^\beta,
    \qquad \forall \rho \geq 0,\ \forall 0 < r < 1,
  \end{equation}
  for some constants $C, \beta > 0$ and for some $x_0 \in X$.

  Let $(B_m)_{m=1}^\infty$ be a decreasing sequence of balls with
  centres in $x_0$ and suppose that there is a number $m_0$ such
  that $\mu (B_m) \geq m^{-1} (\log m)^{2 + \varepsilon}$ holds
  for some $\varepsilon > 0$ and all $m \geq m_0$. Then
  $\mu(H_\mathrm{ea}) = 1$.
\end{theorem}

\begin{theorem} \label{the:Holder2}
  Suppose $(X,T,\mu)$ has exponential decay of correlations for
  H\"{o}lder continuous functions and that the measure $\mu$
  satisfies
  \begin{equation} \label{eq:annulus2}
    \mu (B(x_m, r + \rho) \setminus B(x_m,r)) \leq C \rho^\beta,
    \qquad \forall \rho \geq 0,\ \forall 0 < r < 1,
  \end{equation}
  for some constants $C, \beta > 0$ and for a sequence  of points
  $(x_m)_{m=1}^\infty$.

  Let $(B_m)_{m=1}^\infty$ be a decreasing sequence of balls such
  that $x_m$ is the centre of $B_m$. Suppose that $m \mapsto m
  \mu (B_m)$ is increasing, and that there is a number $m_0$ such
  that $\mu (B_m) \geq m^{-1 + \varepsilon}$ holds for some
  $\varepsilon > 0$ and all $m \geq m_0$. Then
  \[
  \lim_{m \to \infty} \frac{\# \{\, 0 \leq k < m : T^k (x) \in
    B_m \, \}}{m \mu (B_m)} = 1.
  \]
\end{theorem}

Note that in Theorem~\ref{the:Holder1} we assume that the balls
have the same centre, but in Theorem~\ref{the:Holder2}, we assume
only that the balls are nested.

\section{Proof of Theorems~\ref{the:BV1} and \ref{the:BV2}} \label{sec:BV}

\subsection{Preparations}

We start with some preparations that will be used in both proofs of
Theorems~\ref{the:BV1} and \ref{the:BV2}. The basic result in this
section, which will be used in both proofs, is the inequality
\eqref{eq:basicineqBV} below.

Since the conclusion of the theorem does not depend on any finite
number of the balls $B_m$, we may assume that $m_0 = 1$ and
$\varepsilon \in (0,1)$.

We let
\[
Z_m (x) = \sum_{k = 0}^{m-1} \mathbbm{1}_{B_m} (T^k x)
= \sum_{k = 0}^{m-1} \mathbbm{1}_{T^{-k} B_m} (x),
\]
so that
\[
Z_m (x) = \# \{\, 0 \leq k < m : T^k (x) \in B_m \, \}.
\]
The proof proceeds in a way similar to how strong Borel--Cantelli
lemmata can be proved, see for instance Galatolo
\cite[Lemma~7]{Galatolo}. However, the proof of Theorem~\ref{the:BV2}
gets more involved than for strong Borel--Cantelli lemmata, since the
terms in the sum $Z_m$ depends not only on $k$ but also on $m$.

By invariance of the measure $\mu$, the expectation of $Z_m$ satisfies
\[
\expectation Z_m := \int Z_m \, \mathrm{d} \mu = m \mu (B_m).
\]
If we put
\[
Y_m = \frac{Z_m}{\expectation Z_m} - 1,
\]
then we have
\[
\expectation (Y_m^2) = \frac{1}{(\expectation Z_m)^2} \bigl(
\expectation (Z_m^2) - (\expectation Z_m)^2 \bigr).
\]

We start by estimating $\expectation (Z_m^2) - (\expectation
Z_m)^2$. From the definition of $Z_m$ follows that
\[
\expectation (Z_m^2) - (\expectation Z_m)^2 = \sum_{j,k = 0}^{m-1}
\biggl( \int \mathbbm{1}_{B_m} \circ T^j \cdot \mathbbm{1}_{B_m} \circ
T^k \, \mathrm{d} \mu - \mu (B_m)^2 \biggr).
\]
We use decay of correlations. By \eqref{eq:intervalcorrelation}, we
have
\begin{align*}
  \expectation (Z_m^2) - (\expectation Z_m)^2 &\leq 6 \sum_{k =
    0}^{m-1} \sum_{j = k}^{m-1} C e^{-\tau (j - k)} \mu (B_m) \\ &
  \leq C_1 \sum_{k=0}^{m-1} \mu (B_m) = C_1 m \mu (B_m).
\end{align*}

Hence, we have
\begin{equation} \label{eq:basicineqBV}
  \expectation (Y_m^2) \leq C_1 \frac{m \mu (B_m)}{m^2 \mu (B_m)^2} =
  \frac{C_1}{m \mu (B_m)}.
\end{equation}
From now on the proofs of Theorems~\ref{the:BV1} and \ref{the:BV2} are
different, and we treat them separately.

\subsection{Proof of Theorem~\ref{the:BV1}} \label{sec:proofBV1}

Let $(\tilde{B}_m)_{m=1}^\infty$ be a sequence of balls such that
$\tilde{B}_m$ has the same centre as $B_m$ and such that $\mu
(\tilde{B}_m) = c (\log m)^{1 + \varepsilon} m^{-1}$ holds for some
constant $c > 0$.

We may form $\tilde{Z}_m$ and $\tilde{Y}_m$ defined
as $Z_m$ and $Y_m$, but with $B_m$ replaced by $\tilde{B}_m$.
Then, the inequality
\[
  \expectation (\tilde{Y}_m^2) \leq \frac{C_1}{m \mu (\tilde{B}_m)}
\]
holds by the same reason as \eqref{eq:basicineqBV}, and we get
\[
\expectation (\tilde{Y}_m^2) \leq \frac{C_1}{c} (\log m)^{-1-\varepsilon}.
\]

We let $m_n = 2^n$, so that $\expectation (\tilde{Y}_{m_n}^2)
\leq \frac{C_1}{c} (n \log 2)^{-1 - \varepsilon}$ is summable. It
follows that almost surely $\tilde{Y}_{m_n} \to 0$ as $n \to
\infty$. In particular, almost surely
\[
\tilde{Z}_{m_n} \geq \frac{1}{2} \expectation \tilde{Z}_{m_n} \geq 1
\]
for all large enough $n$. It follows that for almost all $x$, there is an
$N$ such that for all $n > N$, there is a $k < m_n$ with $T^k (x) \in
\tilde{B}_{m_n}$.

Now, we consider the sequence of balls $(B_m)_{m=1}^\infty$
such that $(B_m)_{m=1}^\infty$ satisfies $\mu (B_m) \geq (\log
m)^{1+\varepsilon} m^{-1}$.  We assume that $x$ is such that an $N$ as
described above exists, which is the case for almost all $x$. Suppose
that $m_N < m_n < m < m_{n+1}$. Then there is a $k < m_n$ such that
$T^k (x) \in \tilde{B}_{m_n}$. Since
\[
\frac{\mu (\tilde{B}_{m_n})}{\mu (\tilde{B}_{m_{n+1}})} = 2 \Bigl(
\frac{n}{n+1} \Bigr)^{1+\varepsilon} \leq 2,
\]
we have $\mu (\tilde{B}_{m_n}) \leq 2 \mu
(\tilde{B}_{m_{n+1}})$. Similary, we have
\[
\frac{\mu (\tilde{B}_{m_{n+1}})}{\mu (B_m)} \leq c
\frac{m}{m_{n+1}} \Bigl( \frac{n+1}{\log m}
\Bigr)^{1+\varepsilon} \leq 2c
\]
if $n$ is large enough. Hence $\mu (\tilde{B}_{m_n}) \leq 4c \mu
(B_m)$ and then, since all balls are centred at $x_0$, we have
\[
\tilde{B}_{m_n} \subset B_m
\]
if we choose $c < \frac{1}{4}$. We then have $T^k (x) \in B_m$.

We have thus proved that for all large enough $m$ there is a $k < m$
such that $T^k (x) \in B_m$. Since this is true for almost every $x$,
we conclude that $H_\mathrm{ea} ((B_m)_{m=1}^\infty)$ is of full
measure, which finishes the proof of Theorem~\ref{the:BV1}.

\subsection{Proof of Theorem~\ref{the:BV2}}

We now assume that $m \mapsto \mu (B_m)$ is decreasing, $m \mapsto m
\mu (B_m)$ is increasing, and that there is a number $m_0$ such that
$\mu(B_m) \geq m^{-1 + \varepsilon}$ holds for some $\varepsilon > 0$
and all $m \geq m_0$.  Then we have by \eqref{eq:basicineqBV} that
\[
\expectation (Y_m^2) \leq C_1 m^{-\varepsilon}
\]
holds.

We may therefore conclude that for almost every $x$, there are
infinitely many $m$ for which $Y_m \approx 0$. However, we want to
prove that this is the case not only for infinitely many $m$ but for
all $m$ which are large enough. To do so we proceed in the following
way.

Put $m_n = \lceil n^\frac{2}{\varepsilon} \rceil$. Then
\[
\expectation (Y_{m_n}^2) \leq C_1 n^{-2},
\]
which is summable in $n$. It follows that for almost all $x$ we have
$Y_{m_n} (x) \to 0$ as $n \to \infty$. In particular, $Y_{m_n} (x)$
and $Y_{m_{n+1}} (x)$ are both small if $n$ is large, and we shall use
this information to prove that if $m$ is a number between $m_n$ and
$m_{n+1}$, then $Y_m (x)$ is small as well.

Suppose that $m_n < m < m_{n+1}$. Since the balls form a decreasing
sequence, we have
\begin{align*}
  Z_m (x) &= \sum_{k=0}^{m_n-1} \mathbbm{1}_{B_m} (T^k x) +
  \sum_{k=m_n}^{m - 1} \mathbbm{1}_{B_m} (T^k x) \\ &\leq
  Z_{m_n} (x) + S_n (x),
\end{align*}
where
\[
S_n (x) = \sum_{k=m_n}^{m_{n+1} - 1} \mathbbm{1}_{B_{m_n}} (T^k x).
\]
Similary, we have
\[
Z_m (x) \geq Z_{m_{n+1}} (x) - S_n (x).
\]

Using these inequalities, we have
\begin{align*}
  Y_m &= \frac{Z_m}{\expectation Z_m} - 1 \leq \frac{Z_{m_n} +
    S_n}{\expectation Z_{m_n}} \frac{\expectation
    Z_{m_n}}{\expectation Z_m} - 1 \\ & = Y_{m_n} \frac{\expectation
    Z_{m_n}}{\expectation Z_m} + \frac{S_n}{\expectation Z_{m_n}}
  \frac{\expectation Z_{m_n}}{\expectation Z_m} + \frac{\expectation
    Z_{m_n}}{\expectation Z_m} - 1
\end{align*}
and
\begin{align*}
  Y_m &= \frac{Z_m}{\expectation Z_m} - 1 \geq \frac{Z_{m_{n+1}} -
    S_n}{\expectation Z_{m_{n+1}}} \frac{\expectation
    Z_{m_{n+1}}}{\expectation Z_m} - 1 \\ & = Y_{m_{n+1}}
  \frac{\expectation Z_{m_{n+1}}}{\expectation Z_m} -
  \frac{S_n}{\expectation Z_{m_{n+1}}} \frac{\expectation
    Z_{m_{n+1}}}{\expectation Z_m} + \frac{\expectation
    Z_{m_{n+1}}}{\expectation Z_m} - 1.
\end{align*}

With $Q_{n,m} = \frac{\expectation Z_{m_n}}{\expectation Z_m}$ we
therefore have
\begin{multline} \label{eq:Yestimate}
  Y_{m_{n+1}} Q_{n+1,m} - \frac{S_n}{\expectation Z_{m_{n+1}}}
  Q_{n+1,m} + Q_{n+1,m} - 1 \\ \leq Y_m \leq Y_{m_n} Q_{n,m} +
  \frac{S_n}{\expectation Z_{m_n}} Q_{n,m} + Q_{n,m} - 1.
\end{multline}

Since $m \mapsto m \mu (B_m)$ is increasing,
\[
Q_{n,m} = \frac{m_n \mu (B_{m_n})}{m \mu (B_m)} \leq 1
\]
and since $m \mapsto \mu (B_m)$ is decreasing,
\[
Q_{n,m} = \frac{m_n \mu (B_{m_n})}{m \mu (B_m)} \geq \frac{m_n}{m}.
\]
Hence $Q_{n,m} \to 1$ as $m \to \infty$. Similarly, we have
\[
Q_{n+1,m} = \frac{m_{n+1} \mu (B_{m_{n+1}})}{m \mu (B_m)} \geq 1
\]
and
\[
Q_{n+1,m} = \frac{m_{n+1} \mu (B_{m_{n+1}})}{m \mu (B_m)} \leq
\frac{m_{n+1}}{m}.
\]
Therefore, $Q_{n+1,m} \to 1$ as $m \to \infty$.

We may now conclude from \eqref{eq:Yestimate} that $Y_m \to 0$ almost
everywhere provided that $\frac{S_n}{\expectation Z_{m_n}},
\frac{S_n}{\expectation Z_{m_{n+1}}} \to 0$ almost everywhere.

Since $\expectation Z_{m_n} = m_n \mu (B_{m_n}) \leq m_{n+1} \mu
(B_{m_{n+1}}) = \expectation Z_{m_{n+1}}$, we always have
\[
0 \leq \frac{S_n}{\expectation Z_{m_{n+1}}} \leq
\frac{S_n}{\expectation Z_{m_n}},
\]
so it is sufficient to prove that $\frac{S_n}{\expectation Z_{m_n}}
\to 0$ almost everywhere. The expectation of $\frac{S_n}{\expectation
  Z_{m_n}}$ is
\[
\frac{\expectation S_n}{\expectation Z_{m_n}} = \frac{m_{n+1} - m_n}{m_n}
= \frac{m_{n+1}}{m_n} - 1 \sim \frac{2}{\varepsilon} \frac{1}{n},
\]
and
\begin{align*}
  \expectation \biggl( \biggl( \frac{S_n}{\expectation Z_{m_n}} & -
  \frac{\expectation S_n}{\expectation Z_{m_n}} \biggr)^2 \biggr) =
  \frac{\expectation (S_n^2) - (\expectation S_n)^2}{(\expectation
    Z_{m_n})^2} \\ &= (m_n \mu(B_{m_n}))^{-2} \sum_{j,k =
    m_n}^{m_{n+1} - 1} \bigl( \mu (T^{-k} B_{m_n} \cap T^{-j} B_{m_n})
  - \mu (B_{m_n})^2 \bigr) \\ & \leq 6 (m_n \mu (B_{m_n}))^{-2}
  \sum_{k=m_n}^{m_{n+1}} \sum_{j = k}^{m_{n+1}} C e^{-\tau (j-k)} \mu
  (B_m) \\ & \leq C_1 m_n^{-2} (m_{n+1} - m_n) \mu (B_{m_n})^{-1} \leq
  C_2 m_n^{-\varepsilon} \leq C_2 n^{-2}.
\end{align*}
The sequence $n^{-2}$ is summable, and it therefore follows that for
almost every $x$,
\[
\frac{S_n (x)}{\expectation Z_{m_n}} \sim \frac{\expectation
  S_n}{\expectation Z_{m_n}} \sim \frac{2}{\varepsilon} \frac{1}{n}
\]
as $n \to \infty$. In particular, it follows that for almost every
$x$, we have $\lim_{n \to \infty} \frac{S_n (x)}{\expectation Z_{m_n}}
= 0$.

We have now proved that $Y_m \to 0$ almost everywhere. This implies
that for almost every $x$
\[
\frac{ \# \{\, 0 \leq k < m : T^k (x) \in B_m \, \}}{m \mu (B_m)}
\to 1, \qquad \text{as } m \to \infty.
\]

\section{Proof of Theorems~\ref{the:Holder1} and \ref{the:Holder2}} \label{sec:Holder}

The proofs of Theorems~\ref{the:Holder1} and \ref{the:Holder2} are
more difficult than those of Theorems~\ref{the:BV1} and \ref{the:BV2},
since the functions $\mathbbm{1}_{B_m}$ are not H\"{o}lder
continuous. We will need to approximate the functions
$\mathbbm{1}_{B_m}$ with H\"{o}lder continuous functions and proceed
as in the proof of Theorems~\ref{the:BV1} and \ref{the:BV2} for the
corresponding approximate versions of $Z_m$.

\subsection{Preparations}

We start with some preparations, that will be useful in both proofs of
Theorems~\ref{the:Holder1} and \ref{the:Holder2}.

As before, since the conclusion of the theorem does not depend on any
finite number of the balls $B_m$, we may assume that $m_0 = 1$ and
$\varepsilon \in (0,1)$.

We let
\[
Z_m (x) = \sum_{k=0}^{m-1} \mathbbm{1}_{B_m} (T^k x).
\]

Let $\rho_m \in (0, r_m)$ and in case of
Theorem~\ref{the:Holder1}, let $x_m = x_0$ for all $m$.  We take
H\"{o}lder continuous function $f_m^-$ and $f_m^+$ such that
$f_m^- \geq f_{m+1}^-$, $f_m^+ \geq f_{m+1}^+$,
\[
\mathbbm{1}_{B(x_m, r_m - \rho_m)} \leq f_m^- \leq \mathbbm{1}_{B(x_m,
  r_m)} = \mathbbm{1}_{B_m} \leq f_m^+ \leq \mathbbm{1}_{B(x_m, r_m +
  \rho_m)},
\]
and such that $\lVert f_m^\pm \rVert \leq 2 \rho_m^{-\alpha}$,
where the norm is the $\alpha$-H\"{o}lder norm.
This can be done for instance by a piecewise
affine interpolation radially around the centre $x_m$ of the ball
$B_m$: Let
\[
f_m^+ (x) = \left\{ \begin{array}{ll} 1 & \text{if } d(x, x_m)
  \leq r_m \\ 1 - \frac{d(x, x_m) - r_m}{\rho_m} & \text{if } r_m
  < d(x, x_m) < r_m + \rho_m \\ 0 & \text{if } d(x, x_m) \geq r_m
  + \rho_m
  \end{array} \right.,
\]
and
\[
f_m^+ (x) = \left\{ \begin{array}{ll} 1 & \text{if } d(x, x_m)
  \leq r_m - \rho_m \\ 1 - \frac{d(x, x_m) - r_m +
    \rho_m}{\rho_m} & \text{if } r_m - \rho_m < d(x, x_m) < r_m
  \\ 0 & \text{if } d(x, x_m) \geq r_m
  \end{array} \right..
\]
In case of Theorem~\ref{the:Holder1}, we have $x_m = x_0$ for all
$m$.

Put
\[
Z_m^- (x) = \sum_{k=0}^{m-1} f_m^- (T^k x), \qquad \text{and} \qquad
Z_m^+ (x) = \sum_{k=0}^{m-1} f_m^+ (T^k x),
\]
so that $Z_m^- \leq Z_m \leq Z_m^+$. For future use, we note that by
\eqref{eq:annulus2} we have
\begin{align}
  m \mu (B_m) - C m \rho_m^\beta & \leq \expectation Z_m^- = m \mu
  (f_m^-) \leq m \mu (B_m), \label{eq:f-measure} \\ m \mu (B_m) & \leq
  \expectation Z_m^+ = m \mu (f_m^+) \leq m \mu (B_m) + C m
  \rho_m^\beta. \label{eq:f+measure}
\end{align}

We put
\[
Y_m^\pm = \frac{Z_m^\pm}{\expectation Z_m^\pm} - 1,
\]
and we have
\[
\expectation ((Y_m^\pm)^2) = \frac{1}{(\expectation Z_m^\pm)^2} \bigl(
\expectation ((Z_m^\pm)^2) - (\expectation Z_m^\pm)^2 \bigr).
\]

Our first step is to prove that almost surely $Y_m^\pm \to 0$ along a
subsequence. This is the content of the following proposition.

\begin{proposition} \label{prop:basicHolder}
  \begin{enumerate}
  \item[i)] Suppose that $\mu (B_m) \geq m^{-1 +
    \varepsilon}$. Let $m_n = [n^{\frac{1}{\varepsilon} + 1}]$
    and $\rho_m = m^{-1/\beta}$. Then almost surely
    \[
    \lim_{n \to \infty} Y_{m_n}^- = \lim_{n \to \infty} Y_{m_n}^+ = 0.
    \]

  \item[ii)] Suppose that $\mu (B_m) \geq m^{-1} (\log m)^{2 +
    \varepsilon}$.  Let $m_n = 2^n$ and $\rho_m = m^{-1/\beta}$. Then
    almost surely
    \[
    \lim_{n \to \infty} Y_{m_n}^- = 0.
    \]
  \end{enumerate}
\end{proposition}

\begin{proof}
  First we prove statement i). Assume that $\mu (B_m) \geq m^{-1
    + \varepsilon}$.
  
  We have
  \begin{multline*}
    \expectation ((Z_m^\pm)^2) - (\expectation Z_m^\pm)^2 \bigr) =
    \sum_{j,k=0}^{m-1} \biggl( \int f_m^\pm \circ T^j \cdot f_m^\pm
    \circ T^k \, \mathrm{d} \mu - \mu (f_m^\pm)^2 \biggr) \\ = m (\mu
    ((f_m^\pm)^2) - \mu (f_m^\pm)^2) + 2 \sum_{0 \leq k < j < m}
    \biggl( \int f_m^\pm \circ T^{j-k} \cdot f_m^\pm \, \mathrm{d} \mu
    - \mu (f_m^\pm)^2 \biggr).
  \end{multline*}
  The sum will be estimated using decay of correlations. We split the
  sum into two parts $\Sigma_1$ and $\Sigma_2$, defined by
  \begin{align*}
    \Sigma_1 &= \sum_{j=1}^{m-1} \sum_{k=0}^{j - [(\log m)^\gamma]}
    \biggl( \int f_m^\pm \circ T^{j-k} \cdot f_m^\pm \, \mathrm{d} \mu
    - \mu (f_m^\pm)^2 \biggr), \\ \Sigma_2 &= \sum_{j=1}^{m-1}
    \sum_{k=j-[(\log m)^\gamma]-1}^{j-1} \biggl( \int f_m^\pm \circ
    T^{j-k} \cdot f_m^\pm \, \mathrm{d} \mu - \mu (f_m^\pm)^2 \biggr),
  \end{align*}
  where $\gamma > 1$ is a parameter.

  We start by estimating $\Sigma_1$. Here we use
  \eqref{eq:correlationhoelder} and obtain
  \[
  \Sigma_1 \leq \sum_{j=1}^{m-1} \sum_{k=0}^{j - [(\log m)^\gamma]} 4
  C e^{-\tau (j-k)} \rho_m^{-2\alpha} \leq C_1 m \rho_m^{-2 \alpha}
  e^{-\tau (\log m)^\gamma}.
  \]

  Finally, to estimate $\Sigma_2$, we use that $f_m^\pm \leq 1$, so
  that $f_m^\pm \circ T^{j-k} \cdot f_m^\pm \leq f_m^\pm$. Hence
  \begin{align*}
    \Sigma_2 & \leq \sum_{j=1}^{m-1} \sum_{k=j-[(\log
        m)^\gamma]-1}^{j-1} \bigl( \mu (f_m^\pm) - \mu(f_m^\pm)^2
    \bigr) \\ & \leq \sum_{j=1}^{m-1} \sum_{k=j-[(\log
        m)^\gamma]-1}^{j-1} \mu (f_m^\pm) \leq m (\log m)^\gamma \mu
    (f_m^\pm).
  \end{align*}

  Putting these estimates together, we have
  \begin{align*}
    \expectation ((Z_m^\pm)^2) &- (\expectation Z_m^\pm)^2 \bigr) = m
    (\mu ((f_m^\pm)^2) - \mu (f_m^\pm)^2) + 2 (\Sigma_1 + \Sigma_2)
    \\ &\leq C_2 m \Bigl( \mu ((f_m^\pm)^2) + \rho_m^{-2\alpha}
    e^{-\tau (\log m)^\gamma} + (\log m)^\gamma \mu (f_m^\pm) \Bigr) \\
    &\leq C_2 m \Bigl( \mu (f_m^\pm) + \rho_m^{-2\alpha}
    e^{-\tau (\log m)^\gamma} + (\log m)^\gamma \mu (f_m^\pm) \Bigr),
  \end{align*}
  since $(f_m^\pm)^2 \leq f_m^\pm$. Hence
  \begin{align} 
    \expectation ((Y_m^\pm)^2) & \leq C_2 m \frac{\mu (f_m^\pm) +
      \rho_m^{-2 \alpha} e^{-\tau (\log m)^\gamma} + (\log m)^\gamma
      \mu (f_m^\pm)}{(\expectation Z_m^\pm)^2} \nonumber \\ & \leq C_2
    \frac{1}{m \mu (f_m^\pm)} + C_2 \frac{\rho_m^{-2 \alpha} e^{-\tau
        (\log m)^\gamma}}{m \mu (f_m^\pm)^2} + C_2 \frac{(\log
      m)^\gamma}{m \mu (f_m^\pm)}. \label{eq:Ym-estimates}
  \end{align}

  We treat $Y_m^+$ and $Y_m^-$ separately. For $Y_m^-$ we use
  \eqref{eq:f-measure} and get
  \begin{equation} \label{eq:Ym-}
    \frac{\expectation ((Y_m^-)^2)}{C_2} \leq \frac{m^{-1}}{\mu (B_m)
      - C \rho_m^\beta} + \frac{m^{-1} \rho_m^{-2 \alpha} e^{-\tau
        (\log m)^\gamma}}{( \mu (B_m) - C \rho_m^\beta )^2} +
    \frac{m^{-1} (\log m)^\gamma}{\mu (B_m) - C \rho_m^\beta}.
  \end{equation}

  Take $\rho_m = m^{-1/\beta}$. We then have
  \[
  \mu (B_m) - C \rho_m^\beta \geq m^{-1+\varepsilon} - C m^{-1} \geq
  C_3 m^{-1 + \varepsilon}
  \]
  for some constant $C_3$. Hence
  \[
  \expectation ((Y_m^-)^2) \leq \frac{C_2}{C_3} m^{- \varepsilon} +
  \frac{C_2}{C_3^2} m^{-1 + 2 \frac{\alpha}{\beta} + 2 - 2
    \varepsilon} e^{- \tau (\log m)^\gamma} + \frac{C_2}{C_3} (\log
  m)^\gamma m^{-\varepsilon}.
  \]
  Since $\gamma > 1$, we therefore have
  \[
  \expectation ((Y_m^-)^2) \leq C_4 m^{-\varepsilon} (\log m)^\gamma
  \]
  for some constant $C_4$.

  With $m_n = [n^{\frac{1}{\varepsilon} + 1}]$ we get
  \[
  \expectation ((Y_{m_n}^-)^2) \leq C_5 n^{-1 - \varepsilon} (\log
  n)^\gamma
  \]
  and $\expectation ((Y_{m_n}^-)^2)$ is summable over $n$. Hence we
  have almost surely that $Y_{m_n}^- \to 0$ as $n \to \infty$.

  It remains to treat $Y_m^+$, which is done in a similar way as
  $Y_m^-$. Going back to \eqref{eq:Ym-estimates} we get similarly to
  \eqref{eq:Ym-} that
  \[
  \frac{\expectation ((Y_m^+)^2)}{C_2} \leq \frac{m^{-1}}{\mu (B_m)} +
  \frac{m^{-1} \rho_m^{-2 \alpha} e^{-\tau (\log m)^\gamma}}{\mu
    (B_m)^2} + \frac{m^{-1} (\log m)^\gamma}{\mu (B_m)}.
  \]
  Since $\gamma > 1$ we conclude as for $Y_m^-$ that
  \[
  \expectation ((Y_m^+)^2) \leq C_4 m^{-\varepsilon} (\log m)^\gamma.
  \]
  With $m_n = [n^{\frac{1}{\varepsilon} + 1}]$, as for $Y_m^-$
  this implies that $\expectation ((Y_{m_n}^+)^2)$ is summable
  over $n$, and that almost surely $Y_{m_n}^+ \to 0$ as $n \to
  \infty$. This finishes the proof of statement i).

  We now prove statement ii). Assume that $\mu (B_m) \geq m^{-1} (\log
  m)^{2 + \varepsilon}$.

  We have by \eqref{eq:Ym-} that
  \[
    \frac{\expectation ((Y_m^-)^2)}{C_2} \leq \frac{m^{-1}}{\mu (B_m)
      - C \rho_m^\beta} + \frac{m^{-1} \rho_m^{-2 \alpha} e^{-\tau
        (\log m)^\gamma}}{( \mu (B_m) - C \rho_m^\beta )^2} +
    \frac{m^{-1} (\log m)^\gamma}{\mu (B_m) - C \rho_m^\beta}.
  \]
  Take $\rho_m = m^{-1/\beta}$. We then have
  \[
  \mu (B_m) - C \rho_m^\beta \geq m^{-1} (\log m)^{2 + \varepsilon} -
  C m^{-1} \geq C_6 m^{-1} (\log m)^{2 + \varepsilon}
  \]
  for some constant $C_6$. Hence
  \begin{multline*}
  \expectation ((Y_m^-)^2) \leq \frac{C_2}{C_6} (\log m)^{-2 -
    \varepsilon} \\ + \frac{C_2}{C_6^2} m^{-1 + 2 \frac{\alpha}{\beta}}
  (\log m)^{- 2 - \varepsilon} e^{- \tau (\log m)^\gamma} +
  \frac{C_2}{C_6} (\log m)^{\gamma - 2 - \varepsilon}.
  \end{multline*}
  Since $\gamma > 1$, we therefore have
  \[
  \expectation ((Y_m^-)^2) \leq C_7 (\log m)^{\gamma - 2 -
    \varepsilon}
  \]
  for some constant $C_7$. Letting $\gamma = 1 + \frac{1}{2}
  \varepsilon$, we have
  \[
  \expectation ((Y_m^-)^2) \leq C_7 (\log m)^{-1 - \frac{1}{2}
    \varepsilon}.
  \]
  
  We now take $m_n = 2^n$, and get $\expectation ((Y_{m_n}^-)^2) \leq
  C_8 n^{-1 - \frac{1}{2} \varepsilon}$. This is summable over $n$ and
  hence
  \[
  \lim_{n \to \infty} \expectation ((Y_{m_n}^-)^2) = 0. \qedhere
  \]
\end{proof}

\subsection{Proof of Theorem~\ref{the:Holder1}}

Using the second statement of Proposition~\ref{prop:basicHolder}, the
proof of Theorem~\ref{the:Holder1} now runs along the same lines as
that of Theorem~\ref{the:BV1} in Section~\ref{sec:proofBV1}.

Let $(\tilde{B}_m)_{m=1}^\infty$ be a sequence of balls such that
$\tilde{B}_m$ has the same centre as $B_m$ and such that $\mu
(\tilde{B}_m) = c (\log m)^{2 + \varepsilon} m^{-1}$ holds for some
constant $c > 0$.

We may form $\tilde{Z}_m$, $\tilde{Z}_m^-$ and $\tilde{Y}_m^-$
defined as $Z_m$, $Z_m^-$ and $Y_m^-$, but with $B_m$ replaced by
$\tilde{B}_m$.  Then, by Proposition~\ref{prop:basicHolder} and
with $m_n = 2^n$, it follows that almost surely
$\tilde{Y}_{m_n}^- \to 0$ as $n \to \infty$. In particular,
almost surely
\[
\tilde{Z}_{m_n} \geq \tilde{Z}_{m_n}^- \geq \frac{1}{2} \expectation
\tilde{Z}_{m_n}^- \geq 1
\]
for all large enough $n$. It follows that for almost all $x$, there is an
$N$ such that for all $n > N$, there is a $k < m_n$ with $T^k (x) \in
\tilde{B}_{m_n}$.

Now, we consider the sequence of balls $(B_m)_{m=1}^\infty$
such that $(B_m)_{m=1}^\infty$ satisfies $\mu (B_m) \geq (\log
m)^{2+\varepsilon} m^{-1}$.  We assume that $x$ is such that an $N$ as
described above exists, which is the case for almost all $x$. Suppose
that $m_N < m_n < m < m_{n+1}$. Then there is a $k < m_n$ such that
$T^k (x) \in \tilde{B}_{m_n}$. Since
\[
\frac{\mu (\tilde{B}_{m_n})}{\mu (\tilde{B}_{m_{n+1}})} = 2 \Bigl(
\frac{n}{n+1} \Bigr)^{2+\varepsilon} \leq 2,
\]
we have $\mu (\tilde{B}_{m_n}) \leq 2 \mu
(\tilde{B}_{m_{n+1}})$. Similary, we have
\[
\frac{\mu (\tilde{B}_{m_{n+1}})}{\mu (B_m)} \leq c
\frac{m}{m_{n+1}} \Bigl( \frac{n+1}{\log m}
\Bigr)^{2+\varepsilon} \leq 2c
\]
if $n$ is large enough. Hence $\mu (\tilde{B}_{m_n}) \leq 4c \mu
(B_m)$ and then, since all balls are centred at $x_0$, we have
\[
\tilde{B}_{m_n} \subset B_m
\]
if we choose $c < \frac{1}{4}$. We then have $T^k (x) \in B_m$.

We have thus proved that for all large enough $m$ there is a $k < m$
such that $T^k (x) \in B_m$. Since this is true for almost every $x$,
we conclude that $H_\mathrm{ea} ((B_m)_{m=1}^\infty)$ is of full
measure, which finishes the proof of Theorem~\ref{the:Holder1}.

\subsection{Proof of Theorem~\ref{the:Holder2}}
  
We continue as in the proof of Theorem~\ref{the:BV2}. Suppose
that $m \mapsto \mu (B_m)$ is decreasing, $m \mapsto m \mu (B_m)$
is increasing, and that $\mu (B_m) \geq m^{-1 + \varepsilon}$
holds for some $\varepsilon > 0$ and all $m \geq 1$. To prove
Theorem~\ref{the:Holder2}, we do not need to assume that the
balls have the same centres, but we assume that the balls are
nested.

Let $m_n = [n^{\frac{1}{\varepsilon} + 1}]$ and $\rho_m =
m^{-1/\beta}$. Then Proposition~\ref{prop:basicHolder} implies that
almost surely
\[
\lim_{n \to \infty} Y_{m_n}^- = \lim_{n \to \infty} Y_{m_n}^+ = 0.
\]

Suppose that $m_n < m < m_{n+1}$. Since the balls are shrinking, and
$\rho_m$ is decreasing, we have
\begin{align*}
  Z_m^\pm (x) &= \sum_{j=0}^{m_n-1} f_m^\pm (T^j x) +
  \sum_{j=m_n}^{m - 1} f_m^\pm (T^j x) \\ &\leq
  Z_{m_n}^\pm (x) + S_n^\pm (x),
\end{align*}
where
\[
S_n^\pm (x) := \sum_{j=m_n}^{m_{n+1} - 1} f_{m_n}^\pm (T^j x).
\]
Similary, we have
\[
Z_m^\pm (x) \geq Z_{m_{n+1}}^\pm (x) - S_n^\pm (x).
\]

Using these inequalities, we have
\begin{align*}
  Y_m^\pm &= \frac{Z_m^\pm}{\expectation Z_m^\pm} - 1 \leq
  \frac{Z_{m_n}^\pm + S_n^\pm}{\expectation Z_{m_n}^\pm}
  \frac{\expectation Z_{m_n}^\pm}{\expectation Z_m^\pm} - 1 \\ & =
  Y_{m_n}^\pm \frac{\expectation Z_{m_n}^\pm}{\expectation Z_m^\pm} +
  \frac{S_n^\pm}{\expectation Z_{m_n}^\pm} \frac{\expectation
    Z_{m_n}^\pm}{\expectation Z_m^\pm} + \frac{\expectation
    Z_{m_n}^\pm}{\expectation Z_m^\pm} - 1
\end{align*}
and
\begin{align*}
  Y_m^\pm &= \frac{Z_m^\pm}{\expectation Z_m^\pm} - 1 \geq
  \frac{Z_{m_{n+1}}^\pm - S_n^\pm}{\expectation Z_{m_{n+1}}^\pm}
  \frac{\expectation Z_{m_{n+1}}^\pm}{\expectation Z_m^\pm} - 1 \\ & =
  Y_{m_{n+1}}^\pm \frac{\expectation Z_{m_{n+1}}^\pm}{\expectation
    Z_m^\pm} - \frac{S_n^\pm}{\expectation Z_{m_{n+1}}^\pm}
  \frac{\expectation Z_{m_{n+1}}^\pm}{\expectation Z_m^\pm} +
  \frac{\expectation Z_{m_{n+1}}^\pm}{\expectation Z_m^\pm} - 1.
\end{align*}

With $Q_{n,m}^\pm = \frac{\expectation Z_{m_n}^\pm}{\expectation
  Z_m^\pm}$ we therefore have
\begin{multline} \label{eq:Y-estimate}
  Y_{m_{n+1}}^\pm Q_{n+1,m}^\pm - \frac{S_n^\pm}{\expectation
    Z_{m_{n+1}}^\pm} Q_{n+1,m}^\pm + Q_{n+1,m}^\pm - 1 \\ \leq Y_m^\pm
  \leq Y_{m_n}^\pm Q_{n,m}^\pm + \frac{S_n^\pm}{\expectation
    Z_{m_n}^\pm} Q_{n,m}^\pm + Q_{n,m}^\pm - 1.
\end{multline}

By \eqref{eq:f-measure}, we have
\[
Q_{n,m}^- = \frac{\expectation Z_{m_n}^-}{\expectation Z_m^-} \leq
\frac{m_n \mu (B_{m_n})}{m \mu (B_m) - C} \leq \frac{m \mu (B_{m})}{m
  \mu (B_m) - C}
\]
and
\[
Q_{n,m}^- \geq \frac{m_n \mu (B_{m_n}) - C}{m \mu (B_m)} \geq
\frac{m_n \mu (B_{m}) - C}{m \mu (B_{m})} = \frac{m_n}{m} -
\frac{C}{m \mu (B_m)}
\]
when $m$ and $m_n$ are so large that $m \mu(B_m)$ and $m_n \mu
(B_{m_n})$ are larger than $C$.  Using that $m \mu (B_m)$
increases to infinity and that $\frac{m_n}{m} \to 1$, we get that
$Q_{n,m}^- \to 1$ as $m \to \infty$.  Similarly
\[
Q_{n,m}^+ = \frac{\expectation Z_{m_n}^+}{\expectation Z_m^+} \leq
\frac{m_n \mu (B_{m_n}) + C}{m \mu (B_m)} \leq \frac{m \mu (B_{m}) +
  C}{m \mu (B_m)}
\]
and
\[
Q_{n,m}^+ \geq \frac{m_n \mu (B_{m_n})}{m \mu (B_m) + C} \geq
\frac{m_n \mu (B_{m})}{m \mu (B_{m}) + C} = \frac{m_n}{m}
\frac{1}{1 + \frac{C}{m \mu (B_m)}},
\]
implies that $Q_{n,m}^+ \to 1$ as $m \to \infty$.

We also obtain in a similar way that $Q_{n+1,m}^\pm \to 1$ as $m
\to \infty$, and we may conclude from \eqref{eq:Y-estimate} that
$Y_m^- \to 0$ and $Y_m^+ \to 0$ almost everywhere provided that
$\frac{S_n^\pm}{\expectation Z_{m_n}^\pm} \to 0$ and
$\frac{S_n^\pm}{\expectation Z_{m_{n+1}}^\pm} \to 0$ almost
everywhere.

By \eqref{eq:f-measure} and \eqref{eq:f+measure}, we have
$\expectation Z_{m_{n+1}}^\pm \geq \expectation Z_{m_{n}}^\pm -C$
and therefore
\[
0 \leq \frac{S_n^\pm}{\expectation Z_{m_{n+1}}^\pm} \leq
\frac{S_n^\pm}{\expectation Z_{m_n}^\pm - C} =
\frac{S_n^\pm}{\expectation Z_{m_n}^\pm} \frac{1}{1 -
  \frac{C}{\expectation Z_{m_n}^\pm}}.
\]
Since $\expectation Z_{m_{n}}^\pm \to \infty$, we conclude that
$\frac{S_n^\pm}{\expectation Z_{m_n}^\pm} \to 0$ implies that
$\frac{S_n^\pm}{\expectation Z_{m_{n+1}}^\pm} \to 0$.  It is
therefore sufficient to prove that $\frac{S_n^\pm}{\expectation
  Z_{m_n}^\pm} \to 0$ almost everywhere. That this is the case is
the content of the next proposition.

\begin{proposition} \label{prop:Sn}
  Almost everywhere, we have
  \[
  \lim_{n \to \infty} \frac{S_n^-}{\expectation Z_{m_n}^-} = \lim_{n
    \to \infty} \frac{S_n^+}{\expectation Z_{m_n}^+} = 0.
  \]
\end{proposition}

\begin{proof}
  The expectation of $\frac{S_n^\pm}{\expectation Z_{m_n}^\pm}$
  satisfies
  \[
  \frac{\expectation S_n^\pm}{\expectation Z_{m_n}^\pm} =
  \frac{(m_{n+1} - m_n) \expectation f_{m_n}^\pm}{m_n \expectation
    f_{m_n}^\pm} = \frac{m_{n+1}}{m_n} - 1 \sim \Bigl( 1 +
  \frac{1}{\varepsilon} \Bigr)
  \frac{1}{n},
  \]
  and
  \begin{align*}
    \expectation \biggl( \biggl( \frac{S_n^\pm}{\expectation
      Z_{m_n}^\pm} - \frac{\expectation S_n^\pm}{\expectation
      Z_{m_n}^\pm} \biggr)^2 \biggr) &= \frac{\expectation
      ((S_n^\pm)^2) - (\expectation S_n^\pm)^2}{(\expectation
      Z_{m_n}^\pm)^2} \\ &= (m_n \mu(f_{m_n}^\pm))^{-2}
    \mathfrak{S}_n,
  \end{align*}
  where
  \[
  \mathfrak{S}_n^\pm = \sum_{k,j = m_n}^{m_{n+1} - 1} \bigl( \mu
  (f_{m_n}^\pm \circ T^k f_{m_n}^\pm \circ T^j ) - \mu (f_{m_n}^\pm)^2
  \bigr).
  \]
  We split $\mathfrak{S}_n^\pm$ into two parts,
  $\mathfrak{S}_{n,1}^\pm$ consisting of those $k$ and $j$ such
  that $|k - j| \geq [(\log m_n)^\gamma]$, and
  $\mathfrak{S}_{n,2}^\pm$ consisting of the rest.

  For $\mathfrak{S}_{n,1}^\pm$, we have
  \begin{align*}
    \mathfrak{S}_{n,1}^\pm & \leq 2 \sum_{k = m_n}^{m_{n+1} - 1}
    \sum_{j = k + [(\log m_n)^\gamma]}^\infty C e^{-\tau (j-k)}
    \lVert f_{m_n}^\pm \rVert^2 \\ & \leq C_1 (m_{n+1} - m_n)
    e^{- \tau (\log m_n)^\gamma} \rho_{m_n}^{-2 \alpha}.
  \end{align*}
  Since $\mu (f_{m_n}^\pm \circ T^k f_{m_n}^\pm \circ T^j ) \leq \mu
  (f_{m_n}^\pm) \leq \mu (f_{m_n}^+)$, we have
  \begin{align*}
    \mathfrak{S}_{n,2}^\pm & = 2 (m_{n+1} - m_n) (\log m_n)^\gamma \mu
    (f_{m_n}^+) \\ & \leq 2 (m_{n+1} - m_n) (\log m_n)^\gamma (\mu
    (B_{m_n}) + C \rho_{m_n}^\beta).
  \end{align*}
  Hence, since $\gamma > 1$, we have
  \[
  \mathfrak{S}_n^\pm \leq C_5 (m_{n+1} - m_n) (\log m_n)^\gamma \mu
  (B_{m_n}),
  \]
  for some constant $C_5$.  We therefore get that
  \begin{align*}
    \expectation \biggl( \biggl( \frac{S_n^\pm}{\expectation
      Z_{m_n}^\pm} - \frac{\expectation S_n^\pm}{\expectation
      Z_{m_n}^\pm} \biggr)^2 \biggr) &\leq C_5 \frac{(m_{n+1} - m_n)
      (\log m_n)^\gamma \mu (B_{m_n})}{m_n^2 \mu(f_{m_n}^\pm)^2}
    \\ &\leq C_5 \frac{(m_{n+1} - m_n) (\log m_n)^\gamma \mu
      (B_{m_n})}{m_n^2 \mu(f_{m_n}^-)^2} \\ &\leq C_5 \frac{(m_{n+1} -
      m_n) (\log m_n)^\gamma \mu (B_{m_n})}{(m_n \mu (B_{m_n}) - C m_n
      \rho_{m_n}^{\beta})^2}.
  \end{align*}
  Since $\rho_m = m^{-\frac{1}{\beta}}$ we have
  \begin{align*}
    \expectation \biggl( \biggl( \frac{S_n^\pm}{\expectation
      Z_{m_n}^\pm} - \frac{\expectation S_n^\pm}{\expectation
      Z_{m_n}^\pm} \biggr)^2 \biggr) & \leq C_6 \frac{(m_{n+1} - m_n)
      (\log m_n)^\gamma}{m_n^2 \mu (B_{m_n})} \\ & \leq C_6
    \Bigl(\frac{m_{n+1}}{m_n} - 1 \Bigr) m_n^{-\varepsilon} (\log
    m_n)^\gamma \\ & \leq 2 C_6 n^{-1 - \varepsilon} \Bigl( \Bigl(
    \frac{1}{\varepsilon} + 1 \Bigr) \log n \Bigr)^\gamma
  \end{align*}
  for some constant $C_6$. This is summable over $n$, and it therefore
  follows that for almost every $x$,
  \[
  \frac{S_n^\pm (x)}{\expectation Z_{m_n}^\pm} \sim
  \frac{\expectation S_n^\pm}{\expectation Z_{m_n}^\pm} \sim
  \Bigl( 1 + \frac{1}{\varepsilon} \Bigr) \frac{1}{n}
  \]
  as $n \to \infty$. In particular, it follows that for almost every
  $x$, we have $\lim_{n \to \infty} \frac{S_n^- (x)}{\expectation
    Z_{m_n}^-} = \lim_{n \to \infty} \frac{S_n^+ (x)}{\expectation
    Z_{m_n}^+} = 0$.
\end{proof}

By Proposition~\ref{prop:Sn}, it follows that $\lim_{m\to\infty}
Y_m^- (x) = 0$ for almost every $x$. This implies that for almost
every $x$
\[
\liminf_{m \to \infty} \frac{ \# \{\, 0 \leq k < m : T^k (x) \in B_m
  \, \}}{m \mu (B_m)} \geq 1.
\]
Proposition~\ref{prop:Sn} also implies that $\lim_{m\to\infty} Y_m^+
(x) = 0$ for almost every $x$, and this implies that for almost every
$x$
\[
\limsup_{m \to \infty} \frac{ \# \{\, 0 \leq k < m : T^k (x) \in B_m
  \, \}}{m \mu (B_m)} \leq 1.
\]
We have therefore proved that for almost every $x$ we have that
\[
\lim_{m \to \infty} \frac{ \# \{\, 0 \leq k < m : T^k (x) \in B_m \,
  \}}{m \mu (B_m)} = 1.
\]

\end{document}